\documentclass[12pt]{article}

\usepackage{amsmath}
\usepackage{amsthm}
\usepackage{amssymb}

\usepackage{hyperref}

\usepackage{cite}
\usepackage[multiple]{footmisc}

\usepackage{graphicx}
\usepackage{tikz}

\usepackage{enumitem}
\setlist[enumerate]{label={\textnormal{(\roman*)}}}

\newtheorem{theorem}{Theorem}[]
\newtheorem{proposition}[theorem]{Proposition}
\newtheorem{lemma}[theorem]{Lemma}

\theoremstyle{definition}
\newtheorem{conjecture}[theorem]{Conjecture}
\newtheorem*{remark}{Remark}

\renewcommand{\tt}[1]{\textnormal{\texttt{#1}}}

\DeclareMathOperator{\Stab}{Stab}
\DeclareMathOperator{\RT}{RT}

\DeclareMathOperator{\CRTS}{CRT_S}
\DeclareMathOperator{\CRTI}{CRT_I}
\DeclareMathOperator{\CRTW}{CRT_W}
\DeclareMathOperator{\WRT}{WRT}
\newcommand{\circular}[1]{\langle{#1}\rangle}

\usepackage{color}
\newcommand{\red}[1]{\textcolor{red}{#1}}

\title{The Weak Circular Repetition Threshold over Large Alphabets}
\author{Lucas~Mol and Narad~Rampersad\footnote{The work of Narad Rampersad is supported by the Natural Sciences and Engineering Research Council of Canada (NSERC), [funding reference number 2019-04111].}\\
\small Department of Mathematics and Statistics\\
\small The University of Winnipeg\\
\small 515 Portage Ave.\\
\small Winnipeg, MB, Canada\\
\small R3B 2E9\\
\small \{l.mol, n.rampersad\}@uwinnipeg.ca}
\date{}

\begin{document}

\maketitle

\begin{abstract}
\noindent
The \emph{repetition threshold} for words on $n$ letters, denoted $\RT(n)$, is the infimum of the set of all $r$ such that there are arbitrarily long $r$-free words over $n$ letters.  A repetition threshold for circular words on $n$ letters can be defined in three natural ways, which gives rise to the \emph{weak}, \emph{intermediate}, and \emph{strong} circular repetition thresholds for $n$ letters, denoted $\CRTW(n)$, $\CRTI(n)$, and $\CRTS(n)$, respectively.  Currie and the present authors conjectured that $\CRTI(n)=\CRTW(n)=\RT(n)$ for all $n\geq 4$.  We prove that $\CRTW(n)=\RT(n)$ for all $n\geq 45$, which confirms a weak version of this conjecture for all but finitely many values of $n$.


\end{abstract}

\section{Introduction}

Throughout, we use standard definitions and notations from combinatorics on words (see~\cite{LothaireAlgebraic}).  The word $u$ is a factor of the word $w$ if we can write $w=xuy$ for some (possibly empty) words $x,y$.  If at least one of $x,y$ is nonempty, then we say that $u$ is a \emph{proper factor} of $w$.  For a set of words $L$, the word $u$ is a \emph{factor} of $L$ if $u$ is a factor of some word in $L$.

Let $w=w_1w_2\cdots w_k$ be a word, where the $w_i$'s are letters.  A positive integer $p$ is a \emph{period} of $w$ if $w_{i+p}=w_i$ for all $1\leq i\leq k-p$.  In this case, we say that $|w|/p$ is an \emph{exponent} of $w$, and the largest such number is called \emph{the} exponent of $w$.  For a real number $r>1$, a finite or infinite word $w$ is called $r$-free ($r^+$-free) if $w$ contains no factors of exponent greater than or equal to $r$ (strictly greater than $r$, respectively).

Throughout, for every positive integer $n$, let $A_n$ denote the $n$-letter alphabet $\{\tt{1},\tt{2},\dots,\tt{n}\}$.  For every $n\geq 2$, the \emph{repetition threshold} for $n$ letters, denoted $\RT(n)$, is defined by
\[
\RT(n)=\inf\{r>1\colon\ \text{there are arbitrarily long $r$-free words over $A_n$}\}.
\]
Essentially, the repetition threshold describes the border between avoidable and unavoidable repetitions in words over an alphabet of $n$ letters.  The repetition threshold was first defined by Dejean~\cite{Dejean1972}.  Her 1972 conjecture on the values of $\RT(n)$ has now been confirmed through the work of many authors~\cite{Dejean1972,Pansiot1984,MoulinOllagnier1992,MohammadNooriCurrie2007,Carpi2007,CurrieRampersad2009,CurrieRampersad2009Again,CurrieRampersad2011,Rao2011}:
\[
\RT(n)=\begin{cases}
2, &\text{ if $n=2$};\\
7/4, & \text{ if $n=3$};\\
7/5, & \text{ if $n=4$};\\
n/(n-1), & \text{ if $n\geq 5$.}
\end{cases}
\]
The last cases of Dejean's conjecture were confirmed in 2011 by Currie and the second author~\cite{CurrieRampersad2011}, and independently by Rao~\cite{Rao2011}.  However, probably the most important contribution was made by Carpi~\cite{Carpi2007}, who confirmed the conjecture for all but finitely many values of $n$.

Here, we are interested in the notion of a repetition threshold for \emph{circular words} on $n$ letters.  Two (linear) words $x$ and $y$ are said to be \emph{conjugates} if there are words $u$ and $v$ such that $x=uv$ and $y=vu$.  The conjugates of a word $w$ can be obtained by rotating the letters of $w$ cyclically.  For a word $w$, the \emph{circular word} $\circular{w}$ is the set of all conjugates of $w$.  Intuitively, one can think of a circular word as being obtained from a linear word by linking the ends, giving a cyclic sequence of letters.  A \emph{circumnavigation} of the circular word $\circular{w}$ is a word of the form $ava$, where $a$ is a letter and $av$ is a conjugate of $w$.

By the definition of a factor of a set of words, a word is a factor of a circular word $\circular{w}$ if and only if it is a factor of some conjugate of $w$.  As for linear words, a circular word is $r$-free ($r^+$-free) if it has no factors of exponent greater than or equal to $r$ (strictly greater than $r$, respectively).

While the factors of an $r$-free circular word must be $r$-free as linear words, they need not necessarily be $r$-free when taken as circular words.  For example, the $2$-free circular word $\circular{\tt{012021}}$ contains the factor $\tt{0120}$, and the circular word $\circular{\tt{0120}}$ contains the square factor $00$.  This means that several equivalent definitions of the repetition threshold $\RT(n)$ give rise to distinct notions of a repetition threshold for circular words.

In this paper, we will be most interested in the \emph{weak circular repetition threshold} for $n$ letters, denoted $\CRTW(n)$, and defined by
\[
\CRTW(n)=\inf\{\text{\footnotesize $r>1\colon$ there are arbitrarily long $r$-free circular words over $A_n$}\}.
\]
The \emph{intermediate circular repetition threshold} for $n$ letters, denoted $\CRTI(n)$, is defined by
\[
\CRTI(n)=\inf\{\text{\scriptsize $r>1\colon$ there are $r$-free circular words of every sufficiently large length over $A_n$}\},
\]
and the \emph{strong circular repetition threshold} for $n$ letters, denoted $\CRTS(n)$, is defined by
\begin{align*}
\CRTS(n)&=\inf\{\text{\footnotesize $r>1\colon$ there are $r$-free circular words of every length over $A_n$}\}.
\end{align*}
Evidently, we have 
\begin{align}\label{Hierarchy}
\RT(n)\leq \CRTW(n)\leq \CRTI(n)\leq \CRTS(n)
\end{align} for all $n\geq 2$.  Table~\ref{CRTTable} contains all of the confirmed and conjectured values of $\RT(n)$, $\CRTW(n)$, $\CRTI(n)$, and $\CRTS(n)$.  Note in particular that
\[
\CRTW(2)<\CRTI(2)<\CRTS(2),
\]
which shows that the three notions of circular repetition threshold are not equivalent.

\begin{table}[htb]
\renewcommand{\arraystretch}{1.5}
\begin{center}
\begin{tabular}{r | c c c c}
$n$ & $2$ & $3$ & $4$ &  $n\geq 5$\\\hline
$\RT(n)$ & \boldmath{$2$} & \boldmath{$\frac{7}{4}$} & \boldmath{$\frac{7}{5}$} & \boldmath{$\frac{n}{n-1}$}\\
$\CRTW(n)$ & \boldmath{$2$} & \boldmath{$\frac{7}{4}$} & \red{$\frac{7}{5}$} & \red{$\frac{n}{n-1}$}\\
$\CRTI(n)$ & \boldmath{$\frac{7}{3}$} & \boldmath{$\frac{7}{4}$} & \red{$\frac{7}{5}$} & \red{$\frac{n}{n-1}$}\\
$\CRTS(n)$ & \boldmath{$\frac{5}{2}$} & \boldmath{$2$} & \boldmath{$\frac{3}{2}$} & \boldmath{$\frac{\lceil n/2\rceil+1}{\lceil n/2\rceil}$}
\end{tabular}
\end{center}
\caption{The confirmed and conjectured values of the weak, intermediate, and strong circular repetition thresholds.  Confirmed values are in bold font, while conjectured values are in normal font (and coloured red).}
\label{CRTTable}
\end{table}

Through the work of several authors~\cite{AberkaneCurrie2004,Currie2002,Shur2011,Gorbunova2012,CurrieMolRampersad2018}, all values of the strong circular repetition threshold are known:
\[
\CRTS(n)=\begin{cases}
5/2, &\text{ if $n=2$;}\\
2, &\text{ if $n=3$;}\\
\frac{\lceil n/2\rceil+1}{\lceil n/2\rceil}, &\text{ if $n\geq 4$.}
\end{cases}
\]
The values of the weak and intermediate circular repetition thresholds are only known exactly for $n=2$ and $n=3$:
\begin{itemize}
\item $\CRTW(2)=2$ (Thue~\cite{Berstel1995});
\item $\CRTI(2)=7/3$ (Aberkane and Currie~\cite{AberkaneCurrie2005}, Shur~\cite{Shur2011}); and
\item $\CRTI(3)=\CRTW(3)=7/4$ (Shur~\cite{Shur2011}, Currie et al.~\cite{CurrieMolRampersad2018}).
\end{itemize}

From (\ref{Hierarchy}) and the known values of $\RT(n)$ and $\CRTS(n)$, we have
\[
\frac{n}{n-1} \leq \CRTW(n)\leq \CRTI(n)\leq \frac{\lceil n/2\rceil+1}{\lceil n/2\rceil}
\]
for every $n\geq 5$. These are currently the best known bounds on both $\CRTW(n)$ and $\CRTI(n)$ when $n\geq 5$.  Currie and the present authors~\cite{CurrieMolRampersad2018} recently made the following conjecture, which strengthens the second statement of an older conjecture of Shur~\cite[Conjecture~1]{Shur2011}.

\begin{conjecture}\label{conjecture}
For every $n\geq 4$, we have $\CRTI(n)=\CRTW(n)=\RT(n)$.
\end{conjecture}

Here, we prove a weak version (pun intended) of Conjecture~\ref{conjecture} for all but finitely many values of $n$.

\begin{theorem}\label{main}
For every $n\geq 45$, we have $\CRTW(n)=\RT(n)=n/(n-1)$.
\end{theorem}

The layout of the remainder of the paper is as follows.  In Section~\ref{Review}, we summarize the work of Carpi~\cite{Carpi2007} in confirming all but finitely many cases of Dejean's conjecture.  In Section~\ref{Construction}, we establish Theorem~\ref{main} with a construction that relies heavily on the work of Carpi.  We conclude with a discussion of some related notions of repetition threshold for classes of graphs.

We note that the work of Carpi~\cite{Carpi2007} that we rely on in this paper was also instrumental in the recent proof by Currie and the present authors~\cite{CurrieMolRampersad2019} that for every $n\geq 27$, the number of $n/(n-1)^+$-free words of length $k$ over $n$ letters grows exponentially in $k$.  This speaks to the strength of Carpi's results.

\section{Carpi's reduction to \texorpdfstring{\boldmath{$\psi_n$}}{psi}-kernel repetitions}\label{Review}

In this section, let $n\geq 2$ be a fixed integer.  Pansiot~\cite{Pansiot1984} was first to observe that if a word over the alphabet $A_n$ is $(n-1)/(n-2)$-free, then it can be encoded by a word over the binary alphabet $B=\{\tt{0},\tt{1}\}$.  For consistency, we use the notation of Carpi~\cite{Carpi2007} to describe this encoding.  Let $\mathbb{S}_n$ denote the symmetric group on $A_n$, and define the morphism $\varphi_n:B^*\rightarrow \mathbb{S}_n$ by
\begin{align*}
\varphi_n(\tt{0})&=\begin{pmatrix}
\tt{1} & \tt{2} & \cdots & \tt{n-1}
\end{pmatrix}; \text{ and}\\
\varphi_n(\tt{1})&=\begin{pmatrix}
\tt{1} & \tt{2} & \cdots & \tt{n}
\end{pmatrix}.
\end{align*}
Now define the map $\gamma_n:B^*\rightarrow A_n^*$ by
\[
\gamma_n(b_1b_2\cdots b_k)=a_1a_2\cdots a_k,
\]
where 
\[
a_i\varphi_n(b_1b_2\cdots b_i)=\tt{1}
\]
for all $1\leq i\leq k$.  To be precise, Pansiot proved that if a word $\alpha\in A_n^*$ is $(n-1)/(n-2)$-free, then $\alpha$ can be obtained from a word of the form $\gamma_n(u)$, where $u\in B^*$, by renaming the letters.

Let $u\in B^*$, and let $\alpha=\gamma_n(u)$.  Pansiot showed that if $\alpha$ has a factor of exponent greater than $n/(n-1)$, then either the word $\alpha$ itself contains a \emph{short repetition}, or the binary word $u$ contains a \emph{kernel repetition} (see~\cite{Pansiot1984} for details).  Carpi reformulated this statement so that both types of forbidden factors appear in the binary word $u$.  Let $k\in\{1,2,\dots,n-1\}$, and let $v\in B^+$.  Then $v$ is called a \emph{$k$-stabilizing word} (of order $n$) if $\varphi_n(v)$ fixes the points $\tt{1},\tt{2},\dots,\tt{k}$.  Let $\Stab_n(k)$ denote the set of $k$-stabilizing words of order $n$.  The word $v$ is called a \emph{kernel repetition} (of order $n$) if it has period $p$ and a factor $v'$ of length $p$ such that $v'\in \ker(\varphi_n)$ and $|v|>\frac{np}{n-1}-(n-1)$.
Carpi's reformulation of Pansiot's result is the following.

\begin{proposition}[Carpi~{\cite[Proposition~3.2]{Carpi2007}}]
\label{CarpiReform}
Let $u\in B^*$.  If a factor of $\gamma_n(u)$ has exponent larger than $n/(n-1)$, then $u$ has a factor $v$ satisfying one of the following conditions:
\begin{enumerate}
    \item $v\in \Stab_n(k)$ and $0<|v|<k(n-1)$ for some $1\leq k\leq n-1$; or
    \item $v$ is a kernel repetition of order $n$.
\end{enumerate}
\end{proposition}

Now assume that $n\geq 9$, and define $m=\lfloor (n-3)/6\rfloor$ and $\ell=\lfloor n/2\rfloor$.  Carpi~\cite{Carpi2007} defines an $(n-1)(\ell+1)$-uniform morphism $f_n:A_m^*\rightarrow B^*$ with the following extraordinary property.

\begin{proposition}[Carpi~{\cite[Proposition~7.3]{Carpi2007}}]
\label{CarpiShort}
Suppose that $n\geq 27$, and let $w\in A_m^*$. Then for every $k\in\{1,2,\dots,n-1\}$, the word $f_n(w)$ contains no $k$-stabilizing word of length smaller than $k(n-1)$.
\end{proposition}

We note that Proposition~\ref{CarpiShort} was proven by Carpi~\cite{Carpi2007} in the case that $n\geq 30$ in a computation-free manner.  The improvement to $n\geq 27$ stated here was achieved later by Currie and the second author~\cite{CurrieRampersad2009Again}, using lemmas of Carpi~\cite{Carpi2007} along with a significant computer check.

Proposition~\ref{CarpiShort} says that for \emph{every} word $w\in A_m^*$, no factor of $f_n(w)$ satisfies condition (i) of Proposition~\ref{CarpiReform}. 
Thus, we need only worry about factors satisfying condition (ii) of Proposition~\ref{CarpiReform}, i.e., kernel repetitions.  To this end, define the morphism $\psi_n:A_m^*\rightarrow \mathbb{S}_n$ by $\psi_n(v)=\varphi_n(f_n(v))$ for all $v\in A_m^*$.  A word $v\in A_m^*$ is called a \emph{$\psi_n$-kernel repetition} if it has a period $q$ and a factor $v'$ of length $q$ such that $v'\in \ker(\psi_n)$ and $(n-1)(|v|+1)\geq nq-3$.  Carpi established the following result.

\begin{proposition}[Carpi~{\cite[Proposition~8.2]{Carpi2007}}]
\label{CarpiKernel}
Let $w\in A_m^*$.  If a factor of $f_n(w)$ is a kernel repetition, then a proper factor of $w$ is a $\psi_n$-kernel repetition.
\end{proposition}

\begin{remark}
The word ``proper'' in the statement of Proposition~\ref{CarpiKernel} was not in the original statement of Carpi~\cite[Proposition~8.2]{Carpi2007}, but it is easily verified that Carpi's proof actually proves this stronger statement, which will be necessary for our work.
\end{remark}

In other words, if no proper factor of $w\in A_m^*$ is a $\psi_n$-kernel repetition, then no factor of $f_n(w)$ satisfies condition (ii) of Proposition~\ref{CarpiReform}.  Finally, we note that the morphism $f_n$ is defined in such a way that  the kernel of $\psi_n$ has a very simple structure.

\begin{lemma}[Carpi~{\cite[Lemma~9.1]{Carpi2007}}]
\label{Carpi4}
If $v\in A_m^*$, then $v\in \ker(\psi_n)$ if and only if $4$ divides $|v|_a$ for every letter $a\in A_m$.
\end{lemma}

\section{Constructing \texorpdfstring{\boldmath{$n/(n-1)^+$}}{n/(n-1)+}-free circular words over \texorpdfstring{$n$}{n} letters}\label{Construction}

In this section, let $n\geq 27$ be a fixed integer, and let $m=\lfloor (n-3)/6\rfloor$ and $\ell=\lfloor n/2\rfloor$, as in the previous section.  Finally, define $M=4^{m-2}$.  Since $n\geq 27$, we have $m\geq 4$ and $M\geq 16$.

In order to prove Theorem~\ref{main}, we will construct an $n/(n-1)^+$-free circular word of length $M(n-1)(\ell+1)t$ over $A_n$ for every integer $t\geq 1$.  We first show that we can restrict our attention to words over the smaller alphabet $A_m$, just as Carpi did for linear words.  We begin with an analogue of Proposition~\ref{CarpiReform} for circular words.

\begin{lemma}  \label{CircularCarpiReform}
Let $u\in B^*\cap \ker(\varphi_n)$.  If a factor of the circular word $\circular{\gamma_n(u)}$ has exponent larger than $n/(n-1)$, then the circular word $\circular{u}$ has a factor $v$ satisfying one of the following conditions:
\begin{enumerate}
    \item $v\in \Stab_n(k)$ and $0<|v|<k(n-1)$ for some $1\leq k\leq n-1$; or
    \item $v$ is a kernel repetition of order $n$.
\end{enumerate} 
\end{lemma}

\begin{proof}
Let $u=u_1u_2\cdots u_s$, where $u_i\in B$ for all $i\in\{1,\ldots,s\}$.  Let $\gamma_n(u)=a_1a_2\ldots a_s$, where $a_i\in A_n$ for all $i\in\{1,\ldots,s\}$.
First, we claim that for every conjugate $u'=u_j\cdots u_su_1\cdots u_{j-1}$ of $u$, the word $\gamma(u')$ is equal, up to a permutation of the letters, to the corresponding conjugate $a_j\cdots a_sa_1\cdots a_{j-1}$ of $\gamma_n(u)$.  It suffices to show that $\gamma_n(u_2u_3\cdots u_su_1)$ is equal to $a_2a_3\cdots a_ka_1$, up to a permutation of the letters.  By definition of $\gamma_n$, we have
\[
a_i\varphi_n(u_1u_2\cdots u_i)=\tt{1}
\]
for all $i\in\{1,\ldots,s\}$.  Let $\gamma_n(u_2u_3\cdots u_su_1)=b_2b_3\cdots b_sb_1$.  Again, by definition of $\gamma_n$, we have
\[
b_i[\varphi_n(u_1)]^{-1}\varphi_n(u_1u_2\ldots u_{i})=\tt{1}
\]
for all $i\in\{2,\ldots,s\}$.  Finally, using the definition of $\gamma_n$ and the fact that $u\in \ker(\varphi_n)$, we have
\[
b_1[\varphi_n(u_1)]^{-1}\varphi_n(u_1)=b_1[\varphi_n(u_1)]^{-1}\varphi_n(u_1u_2\ldots u_su_1)=\tt{1}.
\]
Thus, we see that $b_i=a_i\varphi_n(u_1)$ for all $i\in\{1,\ldots,s\}$, and this completes the proof of the claim.

Suppose now that a factor $w$ of the circular word $\circular{\gamma_n(u)}$ has exponent larger than $n/(n-1)$.  Then, up to permutation of the letters, the word $w$ is a factor of $\gamma_n(u')$ for some conjugate $u'$ of $u$. By Lemma~\ref{CarpiReform}, the word $u'$ contains a factor $v$ satisfying either condition (i) or condition (ii).  Since $u'$ is a conjugate of $u$, the word $v$ is a factor of $\circular{u}$, and this completes the proof.
\end{proof}

The next lemma is an analogue of Proposition~\ref{CarpiKernel} for circular words.

\begin{lemma}\label{PsiSuffices}
Let $w\in A_m^*\cap\ker(\psi_n)$.  If the circular word $\circular{w}$ contains no $\psi_n$-kernel repetitions, then the circular word $\circular{\gamma_n(f_n(w))}$ is $n/(n-1)^+$-free.
\end{lemma}

\begin{proof}
Suppose, towards a contradiction, that $\circular{w}$ contains no $\psi_n$-kernel repetitions, and that $\circular{\gamma_n(f_n(w))}$ contains a factor of exponent greater than $n/(n-1)$.  Since $w\in\ker(\psi_n)$, we have $f_n(w)\in \ker(\varphi_n)$.  Thus, by Lemma~\ref{CircularCarpiReform}, some factor $v$ of $\circular{f_n(w)}$ satisfies either condition (i) or condition (ii) of Lemma~\ref{CircularCarpiReform}.  Now $v$ must be a factor of $f_n(w')$ for some circumnavigation $w'$ of $w$ (i.e., we have $w'=aw''a$, where $a\in A_m$ and $aw''$ is a conjugate of $w$).  By Proposition~\ref{CarpiShort}, for every $1\leq k\leq n-1$, the word $f_n(w')$ contains no $k$-stabilizing words of length less than $k(n-1)$.  So it must be the case that $v$ is a kernel repetition of order $n$.  Then by Proposition~\ref{CarpiKernel}, some proper factor of $w'$ must be a $\psi_n$-kernel repetition.  Since every proper factor of the circumnavigation $w'$ is a factor of the circular word $\circular{w}$, this is a contradiction.  So we conclude that the circular word $\circular{\gamma_n(f_n(w))}$ is $n/(n-1)^+$-free.  
\end{proof}

By Lemma~\ref{PsiSuffices}, in order to construct an $n/(n-1)^+$-free circular word of length $M(n-1)(\ell+1)t$ over $A_n$, it suffices to construct a word of length $Mt$ over $A_m$ that lies in $\ker(\psi_n)$ and contains no $\psi_n$-kernel repetitions.  Our construction of such a word uses many ideas of Carpi~\cite[Section~9]{Carpi2007}, and we recommend that the reader reviews this section before proceeding.

Following Carpi, define $\beta=(b_i)_{i\geq 1}$, where
\[
b_i=\begin{cases}
\tt{1}, &\text{ if $i\equiv 1\pmod{3}$;}\\
\tt{2}, &\text{ if $i\equiv 2\pmod{3}$;}\\
b_{i/3}, &\text{ if $i\equiv 0\pmod{3}$.}
\end{cases}
\]
We note that $\beta$ can also be defined (c.f.~\cite[Example~4]{CassaigneKarhumaki1997}) as the fixed point of the morphism $\tau:A_2^*\rightarrow A_2^*$ defined by
\begin{align*}
    \tt{1}&\mapsto \tt{121}\\
    \tt{2}&\mapsto \tt{122}.
\end{align*}
We need two lemmas concerning the factors of $\beta$.

\begin{lemma}\label{BeginEnd}
For every $k\geq 1$, the word $\beta$ has a factor of length $k$ that begins and ends in the letter $\tt{2}$.
\end{lemma}

\begin{proof}
The statement can be verified by a case-based proof depending on the value of $k\bmod{3}$, or by using the automatic theorem proving software \tt{Walnut}~\cite{Walnut}.  We describe the latter approach.  
\begin{figure}
\begin{center}
\includegraphics[scale=0.5]{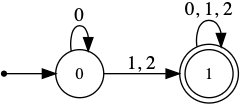}
\end{center}
\caption{The automaton accepting those $(k)_3$ such that the word $\beta$ has a factor of length $k$ that begins and ends in the letter $\tt{2}$.}
\label{BeginsAndEndsInTwo}
\end{figure}
After saving the automaton generating the fixed point of $\tau$ in the \tt{Word Automata Library} folder as \tt{B.txt}, we use the predicate
\[
\text{\footnotesize\tt{eval BeginsAndEndsInTwo "?msd\_3 (k>=1 \& (Ei (B[i]=@2 \& B[i+k-1]=@2)))":}}
\]
The automaton for this predicate is illustrated in Figure~\ref{BeginsAndEndsInTwo}.  The automaton clearly accepts all $(k)_3$ such that $k\geq 1$.
\end{proof}

\begin{lemma}[{Carpi~\cite[Lemma~9.2]{Carpi2007}}]\label{CarpiBeta}
Let $u$ be a factor of $\beta$ with period $q$.  For every $k\geq 0$, if $|u|\geq q+3^k$, then $3^k$ divides $q$.
\end{lemma}

For every $t\geq 1$, we define a set of words $\Lambda_t\subseteq A_m^*$, all of length $Mt$.  We define $\Lambda_t$ by $x_0x_1\cdots x_{Mt-1}\in \Lambda_t$ if and only if $x_i=\max\{a\in A_m\colon\ 4^{a-3} \text{ divides } i\}$ whenever $i\equiv 0$ (mod $4$), and $x_i\in \{\tt{1},\tt{2},\tt{3}\}$ whenever $i\not\equiv 0$ (mod $4$).  

We prove two lemmas concerning the words in $\Lambda_t$.  The proof of the first lemma uses the main idea from Carpi's proof of~\cite[Lemma~9.3]{Carpi2007}.

\begin{lemma}\label{Analogous}
Let $x\in \Lambda_t$, and let $v$ be a factor of the circular word $\circular{x}$.  If $v\in\ker(\psi_n)$, then $M$ divides $|v|$.
\end{lemma}

\begin{proof}
The statement is trivially true if $v=\varepsilon$, so assume $|v|>0$.  Set $|v|=4^bc$, where $4^b$ is the maximal power of $4$ dividing $|v|$.  Suppose, towards a contradiction, that $b\leq m-3$.  Since $v\in\ker(\psi_n)$, by Lemma~\ref{Carpi4}, we see that $4$ divides $|v|$, meaning $b\geq 1$.

Write $x=x_0x_1\cdots x_{Mt-1}$.  Then we have $v=x_ix_{i+1}\cdots x_{i+4^bc-1}$ for some $i\geq 1$, with indices taken modulo $Mt$.  Since $Mt$ is divisible by $M=4^{m-2}$ and $b<m-2$, for all $j\geq 0$ and $k\in\{1,2,\ldots,b\}$, we see that $4^k$ divides $j$ if and only if $4^k$ divides $j \bmod{Mt}$. Since $b\geq 1$, we have $b+3\geq 4$, and hence $x_j\geq b+3$ implies $j\equiv 0$ (mod $4$).  Thus, by definition, for any $j\in \{i,i+1,\ldots,i+4^bc-1\}$, we have $x_j\geq b+3$ if and only if $4^b$ divides $j$.  Thus, we have that the sum $\sum_{a=b+3}^{m}|v|_a$ is exactly the number of integers in the set $\{i,i+1,\dots,i+4^bc-1\}$ that are divisible by $4^b$, which is exactly $c$.  Since $v\in \ker(\psi_n)$, by Lemma~\ref{Carpi4}, we conclude that $4$ divides $c$, contradicting the maximality of $b$.
\end{proof}

\begin{lemma}\label{4Divides4}
Let $x\in \Lambda_t$.  For every letter $a\in\{\tt{4},\tt{5},\ldots,\tt{m}\}$, the number $|x|_a$ is a multiple of $4$.
\end{lemma}

\begin{proof}
Write $x=x_0x_1\ldots x_{Mt-1}$, and let $a\in \{\tt{4},\tt{5},\ldots,\tt{m}\}$.
Since $x\in \Lambda_t$, we have $x_i\geq a$ if and only if $4^{a-3}$ divides $i$.  Since $|x|=Mt=4^{m-2}t$, we have
\[
\sum_{b=a}^m|x|_b=Mt/4^{a-3}=4^{1+m-a}t.
\]
Since $m-a+1\geq 1$, we see that $4$ divides $\sum_{b=a}^m|x|_b$.  The fact that $4$ divides $|x|_a$ now follows by a straightforward inductive argument.
\end{proof}

We are now ready to construct a circular word of length $Mt$ over $A_m$ that belongs to $\ker(\psi_n)$ and contains no $\psi_n$-kernel repetitions.  The proof of the next proposition uses some arguments that were first used by Carpi~\cite[Proposition~9.4]{Carpi2007}.

\begin{proposition}\label{LongPsi}
Suppose that $n\geq 45$.  For every integer $t\geq 1$, there is a word $w\in A_m^*\cap\ker(\psi_n)$ of length $Mt$ such that the circular word $\circular{w}$ contains no $\psi_n$-kernel repetitions.
\end{proposition}

\begin{proof}
Fix $t\geq 1$.  Since $n\geq 45$, we have $m\geq 7$ and $M=4^{m-2}\geq 4^5$.  We use these facts frequently.  Let $u$ be a factor of $\beta$ of length $Mt/4$ that begins and ends in the letter $\tt{2}$; such a factor is guaranteed to exist by Lemma~\ref{BeginEnd}.  Let $\sigma:A_2^*\rightarrow \{\tt{1},\tt{3}\}^*$ be the morphism defined by
\begin{align*}
\tt{1}&\mapsto \tt{1}\\
\tt{2}&\mapsto \tt{3}.
\end{align*}
Write $u\sigma(u)=u_1u_2\ldots u_{Mt/2}$.

Let $s=4-(|u|_\tt{2}\bmod 4)$, and define $v=v_1v_2\cdots v_{Mt/4}$ by
\[
v_i=\begin{cases}
\tt{3}, &\text{ if $1\leq i\leq s$;}\\
\tt{2}, &\text{ if $s+1\leq i\leq 2s$;}\\
\tt{1}, &\text{ if $2s\leq i\leq Mt/4$}.
\end{cases}
\]
Finally, define $w=w_0w_1\cdots w_{Mt-1}$ by 
\[
w_i=\begin{cases}
\max\{a\in A_m\colon\ 4^{a-3} \text{ divides } i\}, & \text{ if $i\equiv 0$ (mod $4$);}\\
v_{(i+2)/4}, &\text{ if $i\equiv 2$ (mod $4$);}\\
u_{(i+1)/2}, &\text{ if $i$ is odd.}
\end{cases}
\]
Evidently, we have $w\in \Lambda_t$.  We will show that $w\in \ker(\psi_n)$, and that the circular word $\circular{w}$ contains no $\psi_n$-kernel repetitions, thus proving the proposition statement.

First, we show that $w\in \ker(\psi_n)$.  By Lemma~\ref{CarpiKernel}, it suffices to show that $4$ divides $|w|_a$ for all $a\in\{\tt{1},\tt{2},\ldots,\tt{m}\}$.  Since $w\in \Lambda_t$, by Lemma~\ref{4Divides4}, we have that $4$ divides $|w|_a$ for all $a\in\{\tt{4},\tt{5},\ldots,\tt{m}\}$.  Next, note that 
\[
|w|_{\tt{2}}=|u|_{\tt{2}}+s=|u|_\tt{2}+4-(|u|_2\bmod{4}),
\]
which is clearly divisible by $4$.  Since
\[
|w|_{\tt{3}}=|\sigma(u)|_{\tt{3}}+s=|u|_{\tt{2}}+s=|w|_{\tt{2}},
\]
we see that $4$ divides $|w|_{\tt{3}}$ as well.  Finally, we have
\[
|w|_{\tt{1}}=|w|-\sum_{b=2}^m |w|_b=Mt-\sum_{b=2}^m|w|_b.
\]
Since $4$ divides both $M$ and $|w|_b$ for all $b\geq 2$, we conclude that $4$ divides $|w|_{\tt{1}}$.

It remains to show that the circular word $\circular{w}$ contains no $\psi_n$-kernel repetitions.  Suppose towards a contradiction that some factor $X$ of $\circular{w}$ is a $\psi_n$-kernel repetition of period $q$. Then by definition, we have 
\begin{align}\label{KernelIneq}
nq-3\leq (n-1)(|X|+1).
\end{align}
Since $m=\lfloor (n-3)/6\rfloor$, we must also have $n\leq 6m+8$.

First of all, consider the case $|X|\leq q+5$.  Applying this inequality to the right side of (\ref{KernelIneq}) and then rearranging, we obtain 
\[
q\leq 6n-3.
\]
By Lemma~\ref{Analogous}, we must have $q\geq M=4^{m-2}$.  Together with the fact that $n\leq 6m+8$, this gives
\[
4^{m-2}\leq 36m+45.
\]
But this last inequality implies $m\leq 6$, a contradiction.

So we may assume that $|X|\geq q+6$.  Then one has $q+2\cdot 3^k\leq |X|<q+2\cdot 3^{k+1}$ for some integer $k\geq 1$.  Using (\ref{KernelIneq}) and the inequality $|X|+1\leq q+2\cdot 3^{k+1}$, we obtain
\begin{align}\label{qUpper}
    q\leq 2(n-1)3^{k+1}+3\leq 3^k(36m+45).
\end{align}

Deleting the letters of even index in $X$, we obtain a factor $x$ of the circular word $\circular{u\sigma(u)}$.  By Lemma~\ref{Analogous}, we know that $4^{m-2}$ divides $q$, so $q$ is certainly divisible by $2$, and $x$ has period $q/2$.  Since $|X|\geq q+2\cdot 3^k$, we have 
\begin{align}\label{xLength}
|x|\geq \left\lfloor \tfrac{|X|}{2}\right\rfloor \geq q/2+3^k\geq q/2+3. 
\end{align}
Write $x=x'y=yx''$, where $|x'|=|x''|=q/2$.   From (\ref{xLength}), we see that $|y|\geq 3$.  

By construction, we may write $u=\tt{2}u'\tt{2}$ and $\sigma(u)=\tt{3}u''\tt{3}$, where $u'\in\{\tt{1},\tt{2}\}^*$ and $u''\in\{\tt{1},\tt{3}\}^*$.  So we have
\[
u\sigma(u)=\tt{2}u'\tt{2}\tt{3}u''\tt{3}.
\]
Thus, in the circular word $\circular{u\sigma(u)}$, the factors $\tt{23}$ and $\tt{32}$ both appear exactly once.  Further, every factor of length $3$ of $u$ contains a $\tt{2}$, and every factor of length $3$ of $\sigma(u)$ contains a $\tt{3}$.  Therefore, either every appearance of $y$ in $\circular{u\sigma(u)}$ is contained entirely in $u$, or every appearance of $y$ in $\circular{u\sigma(u)}$ is contained entirely in $\sigma(u)$.  Without loss of generality, assume that every appearance of $y$ in $\circular{u\sigma(u)}$ is contained entirely in $u$.  Now we consider two cases.

\smallskip

\noindent
\textbf{Case 1:} The factor $x$ is contained entirely in $u$.  

\noindent
Then $x$ is a factor of $\beta$ with period $q/2$.  By Lemma~\ref{Analogous}, we know that $4^{m-2}$ divides $q$.  Recall from (\ref{xLength}) that $|x|\geq q/2+3^k$.  Thus, by Lemma~\ref{CarpiBeta}, we have that $3^k$ divides $q/2$.  Therefore, we have $q\geq 4^{m-2}3^k$.  Together with (\ref{qUpper}), this gives
\[
4^{m-2}\leq 36m+45.
\]
But we have already seen that this inequality implies $m\leq 6$, a contradiction.

\smallskip

\noindent
\textbf{Case 2:} The factor $x$ contains the entire word $\sigma(u)$ as a factor.  

\noindent
Then we can write $x=yzy$, where $|yz|=q/2$, and $z$ contains the entire word $\sigma(u)$.  Let $yzyz'$ be a conjugate of $u\sigma(u)$.  See Figure~\ref{Circle} for an illustration of the circular word $\circular{u\sigma(u)}$.

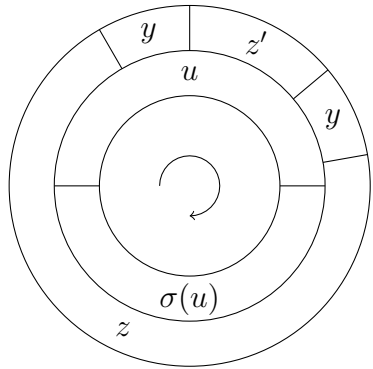
\begin{figure}[htb]
\begin{center}
\begin{tikzpicture}[scale=0.8]
\pgfmathsetmacro{\i}{1.5}
\pgfmathsetmacro{\m}{2.25}
\pgfmathsetmacro{\o}{3}
\pgfmathsetmacro{\im}{(\i+\m)/2}
\pgfmathsetmacro{\mo}{(\m+\o)/2}
\draw (0,0) circle (\i);
\draw (0,0) circle (\m);
\draw (0,0) circle (\o);
\draw (\i,0) -- (\m,0);
\draw (-\i,0) -- (-\m,0);
\draw (10:\m) -- (10:\o);
\draw (40:\m) -- (40:\o);
\draw (90:\m) -- (90:\o);
\draw (120:\m) -- (120:\o);
\node at (105:\mo) {$y$};
\node at (65:\mo) {$z'$};
\node at (25:\mo) {$y$};
\node at (245:\mo) {$z$};
\node at (90:\im) {$u$};
\node at (270:\im) {$\sigma(u)$};
\draw [domain=-180:90,style=->] plot ({0.5*cos(\x)}, {-0.5*sin(\x)});
\end{tikzpicture}
\end{center}
\caption{The circular word $\circular{u\sigma(u)}$ in Case 2 of the proof of Proposition~\ref{LongPsi}.}
\label{Circle}
\end{figure}

Now let $YZYZ'$ be the conjugate of $w$ corresponding to $yzyz'$, i.e., we have that $YZ$ is the length $q$ prefix of $X$, $|Y|=2|y|$, $|Z|=2|z|$, and $|Z'|=2|z'|$.  We claim that $YZ'Y$ is a $\psi_n$-kernel repetition with period $|YZ'|$.   First of all, note that since $yz'y$ is a factor of $u$, while $z$ contains all of $\sigma(u)$, we have $|Z|>|Z'|$.  Next, note that both $YZ\in \ker(\psi_n)$ and $YZYZ'\in\ker(\psi_n)$.  Now let $a\in A_m$.  By Lemma~\ref{CarpiKernel}, both $|YZ|_a$ and $|YZYZ'|_a$ are multiples of $4$. It follows that
\[
|YZ'|_a=|YZYZ'|_a-|YZ|_a
\]
is also a multiple of $4$, and hence $YZ'\in\ker(\psi_n)$.  By the definition of $\psi_n$-kernel repetition, we have
\[
(n-1)(|YZY|+1)\geq n\cdot |YZ|-3.
\]
Rearranging the above inequality, and then using the fact that $|Z|>|Z'|$, we have
\[
(n-2)|Y|+n+2\geq |Z|>|Z'|,
\]
which implies that $YZ'Y$ is also a $\psi_n$-kernel repetition.  But this is impossible by Case 1.
\end{proof}

Finally, Theorem~\ref{main} follows directly from the next result.

\begin{proposition}
Suppose that $n\geq 45$.  For every integer $t\geq 1$, there is a word $W\in A_n^*$ of length $M(n-1)(\ell+1)t$ such that the circular word $\circular{W}$ is $n/(n-1)^+$-free. 
\end{proposition}

\begin{proof}
Fix $t\geq 1$.  By Proposition~\ref{LongPsi}, there is a word $w\in A_m^*\cap \ker(\psi_n)$ of length $Mt$ such that the circular word $\circular{w}$ contains no $\psi_n$-kernel repetitions.  Since $f_n$ is $(n-1)(\ell+1)$-uniform, and $\gamma_n$ preserves length, we have $|\gamma_n(f_n(w))|=M(n-1)(\ell+1)t$.  By Lemma~\ref{PsiSuffices}, the circular word $\circular{\gamma_n(f_n(w))}$ is $n/(n-1)^+$-free.
\end{proof}

\section{Conclusion}

We have shown that $\CRTW(n)=\RT(n)=n/(n-1)$ for all $n\geq 45$.  The conjecture that $\CRTW(n)=\RT(n)$ remains open for all $4\leq n\leq 44$, and the stronger conjecture that $\CRTI(n)=\RT(n)$ remains open for all $n\geq 4$.

To conclude, we will place the notion of weak circular repetition threshold in a broader context, and discuss a more general problem.  Let $G$ be a graph, and let $f:V(G)\rightarrow A_n$ be an $n$-colouring of $G$.  A word $w\in A_n^*$ is called a \emph{factor} of $G$ if $w=f(v_1)f(v_2)\cdots f(v_k)$ for some path $v_1,v_2,\ldots v_k$ in $G$ that contains no repeated vertices.
An $n$-colouring of the graph $G$ is called $r$-free if it contains no factors of exponent greater than or equal to $r$.  By this definition, the $2$-free colourings of $G$ are exactly the \emph{nonrepetitive colourings} of $G$, first defined by Alon et al.~\cite{Alon2002}.  Nonrepetive colourings have been widely studied in the last two decades.  In particular, Dujmovi\'{c} et al.~\cite{Dujmovic2019} recently confirmed what was probably the most important conjecture on nonrepetitive colourings, namely that every planar graph can be nonrepetitively coloured with a bounded number of colours.  Their paper also contains an extensive list of references to other work on nonrepetitive colourings and related notions.

To date, most work on $r$-free colourings has concerned the problem of fixing a number $r$ (most commonly $r=2$) and determining the minimum number of colours necessary for an $r$-free colouring of a given graph. Ochem and Vaslet~\cite{OchemVaslet2012} introduced a notion of repetition threshold for classes of graphs, which considers the problem the other way around -- for a fixed number of colours, find the smallest value of $r$ such that there is an $r$-free colouring of a given graph.  Formally, the \emph{repetition threshold} for $n$ letters on $G$, denoted $\RT(n,G)$, is defined by:
\[
\RT(n,G)=\inf\{\text{$r>1\colon$ there is an $r$-free $n$-colouring of $G$}\}.
\]
For a collection of graphs $\mathcal{G}$, the repetition threshold for $\mathcal{G}$ is defined by $\RT(n,\mathcal{G})=\sup_{G\in\mathcal{G}} \RT(n,G)$.   Note that the strong circular repetition threshold is equivalent to the repetition threshold $\RT(n,\mathcal{C})$, where $\mathcal{C}$ is the collection of cycles.

Ochem and Vaslet determined all values of $\RT(n,\mathcal{T})$, where $\mathcal{T}$ is the collection of all trees.  Lu\v{z}ar, Ochem, and Pinlou~\cite{LuzarOchemPinlou2018} determined all values of $\RT(n,\mathcal{CP})$ and $\RT(n,\mathcal{CP}_3)$, where $\mathcal{CP}$ is the collection of all caterpillars, and $\mathcal{CP}_3$ is the collection of all caterpillars of maximum degree $3$.  They also gave upper and lower bounds on $\RT(n,\mathcal{T}_3)$, where $\mathcal{T}_3$ is the collection of all trees of maximum degree $3$.

Ochem and Vaslet~\cite{OchemVaslet2012} also defined a notion of repetition threshold for ``sufficiently large subdivisions'' of all graphs. For a graph $G$, let $\mathcal{S}(G)$ denote the collection of \emph{subdivisions} of $G$ (i.e., those graphs obtained from $G$ by a sequence of edge subdivisions).  For a collection of graphs $\mathcal{G}$, we define the \emph{weak repetition threshold} for $\mathcal{G}$, denoted $\WRT(n,\mathcal{G})$, by
\[
\WRT(n,\mathcal{G})=\sup_{G\in \mathcal{G}}\inf_{\text{$G_s\in \mathcal{S}(G)$}} \RT(n,G_s).
\] 
The repetition threshold for subdivided graphs, as defined by Ochem and Vaslet, is then equivalent to the weak repetition threshold for the collection $\mathcal{H}$ of all graphs.  Ochem and Vaslet proved that
\[
\WRT(n,\mathcal{H})=
\begin{cases}
7/3, &\text{if $n=2$;}\\
7/4, &\text{if $n=3$;}\\
3/2, &\text{if $n\geq 4$.}\\
\end{cases}
\] 
For all $n\geq 4$, the lower bound $\WRT(n,\mathcal{H})\geq 3/2$  follows from the somewhat trivial fact that any $n$-colouring of a graph with a vertex of degree $n$ must contain a factor of exponent $3/2$.  This suggests restricting to classes of graphs with bounded maximum degree, as was done for the repetition thresholds of caterpillars and trees~\cite{LuzarOchemPinlou2018}.  Let $\mathcal{H}_k$ denote the collection of graphs with maximum degree $k$.  For all $n\geq 2$, it is easy to see that we have
\[
\WRT(n,\mathcal{H}_2)=\CRTW(n),
\]
since every graph of maximum degree $2$ is a disjoint union of paths and cycles.  So by Theorem~\ref{main}, we have $\WRT(n,\mathcal{H}_2)=n/(n-1)$ for all $n\geq 45$.  In addition to determining the unknown values of $\WRT(n,\mathcal{H}_2)$, it would be interesting to determine values of $\WRT(n,\mathcal{H}_k)$ for $k\geq 3$.

%\bibliographystyle{amsplaincustom}
%\bibliography{/Users/lgmol/Documents/WeakCRT.bib}

\end{document}